\documentclass[11pt]{amsart}
\usepackage[all]{xy}
\usepackage{a4wide}
\usepackage{amssymb}
\usepackage{amsthm}
\usepackage{amsmath}
\usepackage{amscd,enumitem}
\usepackage{verbatim}
\usepackage{float}
\usepackage{color}
\usepackage[all]{xy}
\usepackage{hyperref}
\usepackage{bbm}
\textheight8.7in \textwidth6.6in \numberwithin{equation}{section}

\theoremstyle{plain}
\newtheorem{theorem}{Theorem}
\numberwithin{theorem}{section}

\newtheorem{proposition}[theorem]{Proposition}

\theoremstyle{definition}
\newtheorem{definition}[theorem]{Definition}

\theoremstyle{remark}
\newtheorem*{remark}{Remark}
\newtheorem*{remarks}{Remarks}

\newenvironment{psmallmatrix}{\left(\begin{smallmatrix}}{\end{smallmatrix}\right)}

\newcommand{\Stab}{\mathrm{Stab}}
\newcommand{\R}{\mathbb{R}}

\newcommand{\Z}{\mathbb{Z}}
\newcommand{\N}{\mathbb{N}}
\newcommand{\C}{\mathbb{C}}

\renewcommand{\H}{\mathbb{H}}

\newcommand{\z}{\mathfrak{z}}

\newcommand{\re}{\operatorname{Re}}
\newcommand{\im}{\operatorname{Im}}

\newcommand{\SL}{{\text {\rm SL}}}

\begin{document}

\title[Number theoretic generalization of the Monster denominator formula]{Number theoretic generalization of the Monster denominator formula}

\author{Kathrin Bringmann, Ben Kane, Steffen L\"obrich, Ken Ono, and Larry Rolen}

\address{Mathematical Institute, University of Cologne, Weyertal 86-90, 50931 Cologne, Germany}
\email{kbringma@math.uni-koeln.de}

\address{Department of Mathematics, University of Hong Kong, Pokfluam, Hong Kong}
\email{bkane@maths.hku.hk}

\address{Mathematical Institute, University of Cologne, Weyertal 86-90, 50931 Cologne, Germany}
\email{steffen.loebrich@uni-koeln.de}

\address{Department of Mathematics and Computer Science, Emory University,
Atlanta, Georgia 30022} \email{ono@mathcs.emory.edu}

\address{Hamilton Mathematics Institute \& School of Mathematics, Trinity College, Dublin 2, Ireland}
\email{lrolen@maths.tcd.ie}

\thanks{The first and third author are supported by the Deutsche Forschungsgemeinschaft (DFG) Grant No. BR 4082/3-1. The second author was supported by grant project numbers 27300314, 17302515, and 17316416 of the Research Grants Council. The fourth author thanks the support of the NSF and the Asa Griggs Candler Fund.
}
\subjclass[2010]{11F03, 11F37, 11F30}

\begin{abstract} 
The {\it denominator formula} for the Monster Lie algebra is the product expansion for the modular function $j(z)-j(\tau)$  in terms of the Hecke system of $\SL_2(\Z)$-modular functions $j_n(\tau)$. This formula can be reformulated entirely number theoretically. Namely, it
is equivalent to the description of the generating function for the $j_n(z)$ as a weight 2 modular form in $\tau$ with a  pole at  $z$.
Although these results rely on the fact that $X_0(1)$ has
genus 0, here we obtain a generalization, framed in terms of {\it polar harmonic Maass forms}, for all of the $X_0(N)$ modular curves. 
In this survey of recent work, we discuss this generalization, and we offer an introduction to the theory of polar harmonic Maass forms. We conclude with applications to formulas of Ramanujan and Green's functions.
\end{abstract}

\maketitle

\section{Introduction and statement of results}

The theory of {\it Monstrous Moonshine} (see \cite{Moonshine} and \cite{Gannon} for an introduction) offers the {\it denominator formula} for the Monster Lie algebra, the striking infinite product identity
$$
J(z)-J(\tau)
=e^{-2\pi iz} \prod_{m>0, ~n\in \Z}
\left(1-e^{2\pi i m z}e^{2\pi i n \tau}\right)^{c(mn)},
$$
where $J(\tau)$ the usual $\SL_2(\Z)$ Hauptmodul 
\begin{align*}
{{
J(\tau)
}\rm}
=\sum_{n=-1}^{\infty}c(n)e^{2\pi i n \tau}
	&:=\frac{\left(1+240\sum_{n=1}^{\infty}\sum_{d\mid n}d^3 e^{2\pi i n \tau}\right)^3}{e^{2\pi i \tau}\prod_{n=1}^{\infty}(1-e^{2\pi i n \tau})^{24}}
{{
-744
}\rm}
=e^{-2\pi i \tau}+196884e^{2\pi i \tau}+\cdots.
\end{align*}
This identity is the crucial device 
for associating an algebra structure to  the infinite dimensional monster module $V^{\natural}$ whose graded dimensions are the Fourier coefficients of $J(\tau)$. 
Generalizations of the Monster denominator formula to Hauptmoduln of certain genus $0$ subgroups have been obtained in \cite{Scheithauer}.

From a number theoretic perspective, it is natural to consider the logarithmic derivative with respect to $\tau$ of this infinite product.
Asai, Kaneko, and Ninomiya (see Theorem 3 of \cite{AKN}) considered this line of reasoning, and they proved the equivalent formulation
\begin{equation*}
H_{z}(\tau):= \sum_{n=0}^{\infty} j_n(z)e^{2\pi i n \tau}=-\frac{1}{2\pi i} \frac{J'(\tau)}{J(\tau)-J(z)}=\frac{E_4(\tau)^2E_6(\tau)}{\Delta(\tau)} \frac{1}{J(\tau)-J(z)}.
\end{equation*}
The modular functions $j_n$ form a Hecke system. Namely, if we let
$j_0(\tau):=1$ and $j_1(\tau):=J(\tau)$,  then the others are obtained
by applying the normalized Hecke operator $T(n)$
\begin{equation*}
j_n(\tau):=j_1(\tau) \ | \ T(n).
\end{equation*}
The normalization is chosen so that 
$$
j_n(\tau)=e^{-2\pi i n \tau}+O_n(e^{2\pi i \tau}).
$$
\begin{remark}
These objects are critical for Zagier's \cite{Zagier} seminal paper on traces of singular moduli and the Duncan-Frenkel work \cite{DF} on the Moonshine Tower.
Carnahan \cite{Carnahan} has obtained similar denominator formulas for completely replicable modular functions, and other analogues for certain genus zero groups $\Gamma_1(N)$ have been obtained by Ye in \cite{Ye}. 
\end{remark}

In this survey we consider work on generalizations of these results to more general modular curves.
We consider the functions $H_z$ as objects to generalize to other modular curves, not just those
with genus 0 such as $X_0(1)$. To this end, we note  the features that we hope to find in such a generalization.
One obvious feature is that
if $z\in \H$, then $H_{z}$ is a weight $2$ meromorphic modular form on $\SL_2(\Z)$ with a single pole (modulo $\SL_2(\Z)$) at the point $z$. Another critical feature is that the coefficients are the special values $j_n(z)$ for a fixed Hecke system of modular functions.
Although the proof of these features for the $H_z$ relies on the fact that $X_0(1)$ has genus 0, it turns out that there is indeed a generalization for arbitrary genus.
The authors have obtained these results in \cite{FirstPaper}, and this
extension requires polar harmonic Maass forms.

Here we consider the modular curves $X_0(N)$.  For $n\in\N$, we define a Hecke system of  $\Gamma_0(N)$ harmonic Maass functions $j_{N,n}$ which generalize the $j_n$. An important property we obtain about the $j_{N,n}$ is their growth in the $n$-aspect, which is not easily described in terms of Fourier expansions. Instead, we use ``Ramanujan-like'' expansions, sums of the form
\begin{equation}\label{eqn:Ramanujanlike}
\sum_{\lambda\in \Lambda_{z}}\sum_{(c,d)\in S_{\lambda}} \frac{1}{\lambda^{k}} e\left(-\frac{n}{\lambda}r_{z}(c,d,k)\right) e^{\frac{2\pi n\mathrm{Im}(z)}{\lambda}},
\end{equation}
for certain real numbers $r_{z}(c,d,k)$ (see \eqref{eqn:rcd}), $e(x):=e^{2\pi i x}$, $\Lambda_z$ a lattice in $\R$ (see \eqref{eqn:Lambdazdef}), and $S_{\lambda}$ the set of solutions to $Q_{z}(c,d)=\lambda$ for a certain positive-definite binary quadratic form $Q_z$ (see \eqref{eqn:Slambdadef}).

We construct weight 2 polar harmonic Maass forms $H_{N,z}^*$ which generalize  the $H_{z}$.  We have two cases for the $H_{N,z}^*$ which we consider separately. These two cases correspond to points in the upper half plane and those points which are cusps of $X_0(N).$
The first theorem (cf Theorem 1.1 of \cite{FirstPaper}) summarizes the essential properties of these functions when $z\in \H$.

\begin{theorem}\label{AKNGeneralization}
If $z\in \H$, then $H_{N,z}^*(\tau)$ is the unique (up to normalization) weight 2 polar harmonic Maass form on $\Gamma_0(N)$ which vanishes at all cusps and has a single simple pole at $z$. Moreover, the following are true.
\begin{enumerate}[leftmargin=*, label={\rm(\arabic*)}]
\item
 If $z\in \H$ and $\im(\tau) > \operatorname{max}\{\mathrm{Im}(z), \frac1{\mathrm{Im}(z)}\}$, then we have that
$$
H_{N,z}^*(\tau) =\frac{3}{\pi\left[\SL_2(\Z):\Gamma_0(N)\right] \im(\tau)}+\sum_{n=1}^{\infty} j_{N,n}(z)e^{2\pi i n \tau}.
$$

\item For $\gcd(N,n)=1$, we have
$j_{N,n}(\tau) =j_{N,1}(\tau) \ |\ T(n).$
\item For $n\mid N$, we have
$j_{N,n}(\tau)= j_{\frac{N}{n},1}(n\tau).$

\item  As $n\to \infty$, we have
\[
j_{N,n}\left(\tau\right) = \sum_{\substack{\lambda\in\Lambda_{\tau}\\ \lambda\leq n}}\sum_{(c,d)\in S_{\lambda}} e\left(-\frac{n}{\lambda} r_{\tau}(c,d)\right) e^{\frac{2\pi n \mathrm{Im}(\tau)}{\lambda}} + O_{\tau}(n).
\]
\end{enumerate}
\end{theorem}

\smallskip
\noindent
{\it Some Remarks.}

\noindent
(1) In Theorem \ref{AKNGeneralization} (1), the inequality on $\im(\tau)$ is required for convergence.
\smallskip

\noindent
 (2)   For $N=1$, we have that $j_{1,n}(\tau)=j_n(\tau),$
and that
$H_{1,z}^*(\tau)= H_z(\tau) + \frac{3}{\pi \mathrm{Im}(\tau)}- E_{2}(\tau)$.
 \smallskip


\noindent
(3) Theorem~\ref{AKNGeneralization} (4) gives asymptotics for $j_{N,n}(z)$ in $n$-aspect.

\smallskip
\noindent
(4) If $y\geq \im \left(Mz\right)$ for all $M\in\Gamma_0(N)$, then
\begin{equation*}
j_{N,n}(z)\approx e^{-2\pi inz}+\sum_{\substack{c\geq 1\\ N\mid c}} \sum_{\substack{d\in \Z\\ \gcd(c,d)=1\\ |cz+d|^2=1}} e\left(n\frac{d-a}{c}\right) e^{2\pi in\overline{z}}.
\end{equation*}

\smallskip

We now turn to those $H_{N,\varrho}^*$ where $\varrho$
is a cusp of $X_0(N)$. These functions are compatible with the $H_{N,z}^*$ considered in Theorem~\ref{AKNGeneralization}.  Indeed, since $z\mapsto H_{N,z}^*(\tau)$ is continuous (even harmonic) and $\Gamma_0(N)$-invariant, it follows that
\begin{equation}\label{eqn:Hrhodef}
H_{N,\varrho}^*(\tau):=\lim_{z\to \varrho} H_{N,z}^*(\tau).
\end{equation}
The next result (cf. Theorem 1.2 of \cite{FirstPaper}) summarizes the properties of these functions. We use the Kloosterman sums $K_{i\infty, \varrho}(0,-n;c)$ of \eqref{Krhodef} and the weight $2$ harmonic Eisenstein series $E_{2,N,\varrho}^*(\tau)$ with constant term $1$ at $\varrho$ and vanishing at all other cusps.

\begin{theorem}\label{thm:HtauCusp}
For every cusp $\varrho$ of $\Gamma_0(N)$, we have that $H_{N,\varrho}^*(\tau)=-E_{2,N,\varrho}^*(\tau)$.  Moreover, the following are true.
\begin{enumerate}[leftmargin=*, label={\rm(\arabic*)}]
\item We have
\begin{align*}
	H_{N,\varrho}^* (\tau)
	&\hphantom{:}= \frac{3}{\pi\left[\SL_2(\Z):\Gamma_0(N)\right]\im(\tau)}-\delta_{\varrho,\infty}+\sum_{n=1}^\infty  j_{N,n}(\varrho) e^{2\pi in\tau}, \qquad \text{ with}\\
	j_{N,n}(\varrho)
	&:=\lim_{\tau\to \varrho}j_{N,n}(\tau)=\frac{4\pi ^2 n}{\ell_\varrho}\sum_{\substack{c\geq 1\\ N|c}}\frac{K_{i\infty, \varrho}(0,-n;c)}{c^2},
\end{align*}
	where $\delta_{\varrho,\infty}:=1$ if $\varrho =i\infty$ and $0$ otherwise.
\item  For $\gcd(N,n)=1$, we have
$
j_{N,n}(\varrho) =\lim_{\tau\to \varrho}j_{N,1}(\tau) \ |\ T(n).
$
\item For $n\mid N$, we have
$
j_{N,n}(\varrho)= \lim_{\tau\to \varrho} j_{\frac{N}{n},1}(n\tau).
$
\end{enumerate}
\end{theorem}

\begin{remark}
The 
Fourier expansion in Theorem \ref{AKNGeneralization} (1) is not valid as $z\to i\infty$.
\end{remark}

These results are special cases in weight 2 of the theory of polar harmonic Maass forms. This paper is meant to be an introduction of the
theory with the applications above in mind as motivation for their study. 
In Section~\ref{PolarHarmonicMaass:Section} we provide a general overview of the theory. We offer definitions, essential properties,
and explicit constructions of these forms in some detail. Then in Section~\ref{WeightTwo} we consider the special case of weight 2, and we give
a brief sketches of the proofs of Theorems~\ref{AKNGeneralization} and \ref{thm:HtauCusp}. Finally, in Section~\ref{FinalRemarks} we
summarize recent work on formulas of Ramanujan and Green's functions.

\section{General theory of polar harmonic Maass forms}\label{PolarHarmonicMaass:Section}
We begin with the definition of a polar harmonic Maass form. The difference from usual harmonic Maass forms is that we also allow singularites in $\H$. For more on the theory of harmonic Maass forms, the interested reader is also referred to the survey \cite{KenSurvey} and the forthcoming textbook \cite{AMSBook}. Throughout, we let $\tau=u+iv$, with $u, v\in \R$ (resp.  $z=x+iy$, with $x, y\in \R$). 
As usual, for $M=\left(\begin{smallmatrix}a&b\\c&d\end{smallmatrix}\right)\in \SL_2(\Z)$ and $k\in\Z$, we define $j(M,\tau):=c\tau+d$ and let 
\[
{{
F|_{k}
}}
M(\tau) :=j(M,\tau)^{-k}F\left(\frac{a\tau+b}{c\tau+d}\right)
\]
denote the weight $k$ \begin{it}slash operator\end{it}. 
Whenever we apply the slash operator to a function with multiple variables, we resolve the ambiguity by writing, e.g., $|_{k,\tau} M$.

\begin{definition}\label{PolarHMFDefn}
For $k\in\mathbb{Z}$, a {\it polar harmonic Maass form of weight $k$ on $\Gamma_0(N)$} is a function $F\colon\mathbb{H}\to\mathbb{C}$ which is real-analytic outside a discrete set of points and satisfies the following conditions: 
\begin{itemize}
\item[i)] For every $M\in\Gamma_0(N)$, we have $F|_{k}M=F$.
\item[ii)]
The function $F$ is annihilated by the \begin{it}weight $k$ hyperbolic Laplacian\end{it} 
$$
\Delta_{k}:= -v^2\left(\frac{\partial^2}{\partial u^2} + \frac{\partial^2}{\partial v^2}\right)+i kv \left(\frac{\partial}{\partial u} +i\frac{\partial}{\partial v}\right).
$$
\item[iii)] For every $z\in\mathbb{H}$, there exists an $n\in\mathbb{N}_0$ such that $(\tau-z)^nF(\tau)$ is bounded in some neighborhood of $z$.
\item[iv)] The function $F$ grows at most linearly exponentially at the cusps of $\Gamma_0(N)$.
\end{itemize} 
We denote by $\mathcal{H}_{k}(\Gamma_0(N))$ the space of polar harmonic Maass forms of weight $k$ with respect to $\Gamma_0(N)$.
\end{definition}

\begin{remark}
Analogous definitions may be made for half-integral weight, on arbitrary congruence subgroups, and for forms with multiplier systems.
\end{remark}

Polar harmonic Maass forms of weight $k\leq 0$ have a natural decomposition into holomorphic and non-holomorphic parts. To describe these, we require their Fourier expansion around a cusp $\varrho$ of $\Gamma_0(N)$.   

\begin{proposition}\label{FourierExpansion}
Suppose that the cusp width  of $\varrho$ is $\ell_{\varrho}$ and choose $M_{\varrho}$ such that $M_{\varrho}\varrho=i\infty$. If $k\leq 0$, then the Fourier expansion of $F$ at $\varrho$ has the shape (convergent for $v \gg 0$) 
$$
F_{\varrho}(\tau):=F\big|_{k} M_{\varrho}(\tau)=F_{\varrho}^+(\tau) + F_{\varrho}^-(\tau),
$$
where, for some $c_{F,\varrho}^{\pm}(n)\in\C$, 
$$
F_{\varrho}^+(\tau):=\sum_{n\gg -\infty} c_{F,\varrho}^+(n) e^{\frac{2\pi i n\tau}{\ell_{\varrho}}},\quad
F_{\varrho}^-(\tau):=c_{F,\varrho}^-(0) v^{1-k}+\sum_{\substack{n\ll \infty\\ n\neq 0}} c_{F,\varrho}^-(n) \Gamma\left(1-k,-\frac{4\pi n v}{\ell_{\varrho}}\right)e^{\frac{2\pi i n\tau}{\ell_{\varrho}}},
$$
with the incomplete gamma function $\Gamma(a,w):=\int_{w}^{\infty} t^{a-1} e^{-t} dt$. 
\end{proposition}
\begin{proof}[Sketch of proof]
By condition i) in Definition \ref{PolarHMFDefn}, any $F \in  \mathcal{H}_k(\Gamma_0(N))$ has an expansion of the form
\begin{align}\label{FexpHMFgeneral}
F(\tau) =: \sum_{n\in\mathbb Z} c_{f}(n,v)e\left(nu\right).
\end{align}
If $n=0$, then we have that $c_f(0,v)$ is a linear combination of $1$ and $v^{1-k}$. If $n\neq0$, let $C(2\pi n v) := c_f(n,v)$.  Applying $\Delta_k$ to (\ref{FexpHMFgeneral}), one finds that the Fourier coefficients $C(w)$ satisfy the differential equation
\begin{align} \label{FcoeffHMFDE}
\frac{\partial^2}{\partial w^2}C(w) - C(w) + \frac{k}{w} \left(\frac{\partial}{\partial w}C(w) + C(w)\right) = 0.
\end{align}
Thus, if $n\neq0$, equation (\ref{FcoeffHMFDE}) has two linearly independent solutions, namely $e^{-w}$ and $\Gamma(1-k,-2w)e^{-w}$.   
The restrictions in the summations follow from condition iv) in \ref{PolarHMFDefn}.
\end{proof}

We call $F_{\varrho}^+$ the
\begin{it}holomorphic part\end{it} of $F$ at $\varrho$ 
and $F_{\varrho}^-$ the \begin{it}non-holomorphic part\end{it}. 
We call the terms of the Fourier expansion which grow towards $\varrho$ the \begin{it}principal part at $\varrho$\end{it}.\\

The coefficients $c_{F,\varrho}^-(n)$ are closely related to coefficients of meromorphic modular forms of weight $2-k$.  Indeed, this relationship follows from the fact that the hyperbolic Laplacian splits as
$$
\Delta_k=-\xi_{2-k}\circ \xi_k,
$$
where $\xi_{k}:=2iv^{k} \overline{\frac{\partial}{\partial \overline{\tau}}}$.  If $F$ satisfies weight $k$ modularity, then $\xi_{k}(F)$ is modular of weight $2-k$, and thus $\xi_{k}$ maps weight $k$ polar harmonic Maass forms to weight $2-k$ meromorphic modular forms. It is natural to consider the subspace $\mathcal{H}_{k}^{\operatorname{cusp}}(\Gamma_0(N))\subset \mathcal{H}_{k}(\Gamma_0(N))$ consisting of those $F$ for which $\xi_{k}(F)$ is a cusp form.  \\

We next consider {\it elliptic expansions} of polar harmonic Maass forms. For this purpose, it is natural to first look at elliptic expansions of meromorphic modular forms of weight $k$ (attached to points on the upper half plane), as opposed to the more common Fourier series in $q$ (which are expansions around the cusp at $i\infty$). These are of the shape 
\begin{equation*}\label{eqn:fellexp}
f(\tau)=(\tau-\overline{z})^{-k}\sum_{n\gg -\infty} c_{f,z}(n) X_{z}^n(\tau)
\end{equation*}
and converge if 
\begin{equation}
X_z(\tau) := \frac{\tau-z}{\tau-\overline{z}}
\end{equation}
is sufficiently small.  Just as there are two parts to the Fourier expansion at a cusp, polar harmonic Maass forms have expansions around points in the upper half-plane which break into two pieces. 
 
 \begin{proposition}\label{EllipticExpansion}
A  polar harmonic Maass form $F$ of weight $k\leq 0$ has an expansion around each point $z\in\mathbb H$ of the form $F=F^+_{z}+F^-_{z}$,
where the {\it meromorphic part} $F^+_{z}$ is given by
\begin{equation}\label{mp}
F^+_{z}(\tau):=(\tau-\overline{z})^{-k}\sum_{n\gg-\infty}c_{F,z}^+(n)X_{z}^n(\tau)
\end{equation}
and the {\it non-meromorphic part} $F^-_{z}$ by
\begin{equation}\label{nmp}
F^-_{z}(\tau):=(\tau-\overline{z})^{-k}\sum_{n\ll\infty}c_{F,z}^-(n)\beta_0\left(1-|X_{z}(\tau)|^2;1-k,-n\right)X_{z}^n(\tau)
.
\end{equation}
These expressions converge for $|X_{z}(\tau)|\ll 1$.
Here, we have that
$$
\beta_0\left(w; a,b\right):=\beta\left(w; a,b\right)-\mathcal{C}_{a,b} \hspace{7mm} \text{ with }\hspace{7mm}
\mathcal{C}_{a,b}:=\sum_{\substack{0\leq j\leq a-1\\ j\neq -b}} \binom{a-1}{j}\frac{(-1)^j}{j+b},
$$
where the {\it incomplete $\beta$-function} is defined by
$
\beta({w};a,b):=\int_0^{w} t^{a-1} (1-t)^{b-1} dt.
$
\end{proposition}

We refer to the terms 
in \eqref{mp} and \eqref{nmp} 
which grow as $\tau\to z$ as the \begin{it}principal part of $F$ around $z$\end{it}. 

\begin{remarks}

\ \ \ \newline
\noindent
(1)  If $F\in \mathcal{H}_{k}^{\operatorname{cusp}}(\Gamma_0(N))$, then \eqref{nmp} only runs over $n<0$.

\noindent
(2) If $F\in \mathcal{H}_{k}(\Gamma_0(N))$, then the sums in \eqref{mp} and \eqref{nmp} run only over those $n$ which satisfy $n\equiv -k/2\pmod{e_{N,z}^{-1}}$, where $e_{N,z}:=2/\#\Stab_z(\Gamma_0(N))$.  
\end{remarks}

\begin{proof}[Sketch of proof of Proposition \ref{EllipticExpansion}]
For $w:=X_{z}(\tau)$ sufficiently small, the function
$$
\left( \frac{2iy}{1-w}\right)^{k} F \left( \frac{z-\overline{
z}w}{1-w}\right)=\left(\tau-\overline{z}\right)^{k} F(\tau).
$$
has an expansion of the shape $(w=re^{i\theta})$
\begin{equation*}
\sum_{n\in\Z} a_r(n) e^{in\theta}
\end{equation*}
We then define a function $b_n$ by $a_r(n)=: r^nb_n(r^2)$.  A straightforward but lengthy calculation yields that  
\begin{equation}\label{bdiff}
0=(1-t)t b_n''(t) + \left( n+1 - (k+n+1)t\right) b_n'(t).
\end{equation}
The general solution of \eqref{bdiff} is given by
\[
b_n(t)=c_1\beta(1-t; 1-k, -n)+c_2
\]
with $c_1, c_2\in\C$.  Hence, noting that $\beta(1;1-k,-n)$ exists for $n<0$, we conclude that
\begin{multline*}
\left(\tau-\overline{z}\right)^{k}F(\tau)=\sum_{n\in\Z} a_n w^n - \sum_{n\geq 0} b_n \beta\left(1-r^2;1-k,-n\right) w^n\\
+ \sum_{n<0} b_n \left( \beta(1;1-k,-n)-\beta\left(1-r^2;1-k,-n\right)\right)w^n,
\end{multline*}
for some $a_n$ and $b_n \in\C$. From condition iii) in Definition \ref{PolarHMFDefn}, we obtain that $a_n=0$ if $n<-n_0$ and $b_n=0$ if $n> n_0$.
\end{proof}

Following Bruinier and Funke, for $k\in\N_0$, $g\in S_{k}(\Gamma_0(N))$, and $F\in \mathcal{H}_{2-k}^{\operatorname{cusp}}(\Gamma_0(N))$, we define the \textit{pairing}
\begin{equation*}
\{g, F\}:=\left<g, \xi_{2-k}(F)\right>,
\end{equation*}
where for $g,h\in S_{k}(\Gamma_0(N))$, $\left<g,h\right>$ denotes the standard \begin{it}Petersson inner product\end{it}.  The pairing $\{g,F\}$ was computed for $F$ a harmonic Maass form in \cite{BF}. We recall an extension to the entire space $\mathcal{H}_{2-k}^{\operatorname{cusp}}(\Gamma_0(N))$.
\begin{proposition}[Proposition 6.1 of \cite{BK}]\label{expansionprod}
If $g\in S_{k}(\Gamma_0(N))$ and $\mathcal{F}\in\mathcal{H}_{2-k}^{\operatorname{cusp}}(\Gamma_0(N))$, then
\begin{equation*}
\{g, F\}=\frac{1}{\left[\SL_2(\Z):\Gamma_0(N)\right]}\sum_{n\geq 1}\left(\sum_{z\in\Gamma_0(N)\backslash\H} \frac{\pi  e_{N,z} }{y}c_{F,z}^+\left(-n\right) c_{g,z}\left(n-1\right) +\sum_{\varrho\in\mathcal{S}_N} c_{F,\varrho}^{+}(-n)c_{g,\varrho}(n)\right),
\end{equation*}
where $\mathcal{S}_N$ denotes the set of inequivalent cusps of $\Gamma_0(N)$.
\end{proposition}

For $k>2$, the principal parts of the Fourier expansions around all cusps and the principal parts of the elliptic expansions uniquely determine the form.  Indeed, Proposition \ref{expansionprod} implies that any polar harmonic Maass form $F$ without any singularities must satisfy $\xi_{2-k}(F)=0$ and there are no non-trivial negative-weight holomorphic modular forms. \\

For $z\in\mathbb{H}$, $n\in\mathbb{Z}$, and $k\in\N_{>2}$, the {\it meromorphic elliptic Poincar\'e series} of Petersson are given by
\begin{equation}
\label{Psi}
\Psi_{k, n,N}^z(\tau):=\sum_{M\in\Gamma_0(N)}\left. \left((\tau-\overline{z})^{-k}X_{z}^n(\tau)\right)\right|_{k,\tau}M.
\end{equation}
These functions give rise to natural polar harmonic Maass forms. Namely, we have the following general theorem.

\begin{theorem}\label{Psithm}
The functions $\Psi_{k, n,N}^z$ are weight $k$ meromorphic modular forms which are cusp forms if $n\geq 0$. For $n<0$, they are orthogonal to cusp forms and have principal part around $\tau =z$ equal to 
\[
2e_{N,z} (\tau-\overline{z})^{-k}X_{z}^n(\tau)
.
\]
Furthermore, the functions $z\mapsto y^{-k-n}\Psi_{k, n,N}^z(\tau)$ are polar harmonic Maass forms of weight $-2n-k$. 
\end{theorem}
\begin{remark}
By orthogonal to cusp forms we mean with respect to a regularised inner product.   
\end{remark}
\begin{proof}[Sketch of proof of Theorem \ref{Psithm}]
The series in \eqref{Psi} converges absolutely and locally uniformly in both $\tau$ and $z$ for $k>2$. Thus it can be checked termwise that the functions are meromorphic in $\tau$ and harmonic in $z$. The principal parts come from the summand corresponding to the identity matrix and orthogonality to cusp forms has been proven in Satz 8 of \cite{Petersson}. 
We now use the identity
$$
M\tau -Mz = j(M,\tau)j(M,z)(\tau -z)
$$ 
to rewrite
$$
y^{-k-n}j(M,\tau)^k  (M\tau -\overline{z})^n (M\tau -z)^{-n-k} = j(M,z)^{2n+k}\im (M^{-1}z)^{-k-n}(\tau -M^{-1}\overline{z})^n (\tau -M^{-1}z)^{-n-k}.
$$
Summing over $M$ implies that $y^{-k-n}\Psi_{k, n,N}^z(\tau)$ is modular in $z$. 
\end{proof}

\section{The Special case of weight 2}\label{WeightTwo}


In this section, we construct weight $2$ polar harmonic Maass forms $H_{N,z}^*$ which only have simple poles in $z$ modulo $\Gamma_0(N)$ and decay towards the cusps. This is achieved by using  analytic continuation to obtain analogues of $\Psi_{k,-1,N}^z$ for $k=2$. It turns out that these functions are not meromorphic in $\tau$ anymore. However, they are still polar harmonic Maass forms in $\tau$ and in $z$.  To construct them we define for $z, \tau\in \H$ and $s\in\C$ with $\re(s)>0$
\begin{equation}\label{Pns}
	P_{N, s}(\tau, z):=\sum_{M\in \Gamma_0(N)} \frac{\varphi_s\left(M\tau, z\right)}{j\left(M, \tau\right)^2|j\left(M, \tau\right)|^{2s}}
\end{equation}
with  
\begin{equation*}
	\varphi_s(\tau, z):=\mathrm{Im}(z)^{1+s}(\tau-z)^{-1}(\tau-\overline{z})^{-1}\left|\tau-\overline{z}\right|^{-2s}.
\end{equation*}

\begin{proposition}[Lemma 4.4 of \cite{BK}]\label{anacontin}
The function $P_{N,s}(\tau, z)$ has an analytic continuation denoted by $y\Psi_{2,N}(\tau, z)$ to $s=0$. We have $z\mapsto y\Psi_{2,N}(\tau,z)\in \mathcal{H}_0(\Gamma_0(N))$ and $\tau\mapsto y\Psi_{2,N}(\tau,z)\in \mathcal{H}_2(\Gamma_0(N))$.
\end{proposition}

\begin{remark}
Explicit Fourier expansions for $z\mapsto y\Psi_{2,N}(\tau,z)$ are given in Lemma 5.4 of \cite{BK} and for $\tau\mapsto y\Psi_{2,N}(\tau,z)$ in Proposition 2.1 of \cite{FirstPaper}.
\end{remark}

\begin{proof}[Sketch of proof of Proposition \ref{anacontin}]
 One uses a splitting of the sum in \eqref{Pns} due to Petersson and computes the Fourier expansion of each part. These converge absolutely and locally uniformly in $s$ and can be analytically continued to $s=0$. Once the analytic continuation is established, one can proceed as in the proof of Theorem \ref{Psithm}
\end{proof}

We then set
\begin{equation}\label{eqn:Hdef}
H_{N,z}^*(\tau):=-\frac{y}{2\pi}\Psi_{2,N}(\tau,z).
\end{equation}
Note that $H^*_{N,z}$ has principal part 
$$
\frac{1}{2\pi i e_{N, z}}\frac{1}{\tau-z}
$$ 
at $\tau=z$.\\

The Fourier coefficients $j_{N,n}(z)$ of $H_{N,z}^*$ are given by analytic continuations of {\it Niebur's Poincar\'e series} \cite{Niebur}. To be more precise, set for $n\in \N$ and $\re(s)>1$
\begin{align*}
	F_{N, -n, s}(z)
	:=\sum_{M\in\Gamma_\infty \backslash\Gamma_0(N)}e\left(-n\re (Mz)\right)\im(Mz)^\frac12 I_{s-\frac12}\left(2\pi n \im(Mz)\right)
\end{align*}
with $\Gamma_\infty:=\left\{\pm\begin{psmallmatrix}
	1 & n\\
	0 & 1
\end{psmallmatrix}:n\in\Z\right\}$ and $I_{\alpha}$ the usual $I$-Bessel function of order $\alpha$. These functions are {\it weak Maass forms} of weight $0$; instead of being annihilated by $\Delta_0$, they have eigenvalue $s(1-s)$.

\begin{proposition}[Theorem 1 of \cite{Niebur}]\label{Niebur}
	The function $F_{N, -n, s}$ has an analytic continuation $F_{N, -n}$ to $s=1$ and $F_{N, -n}\in \mathcal H_0 (\Gamma_0(N))$.
	\end{proposition}
 To obtain the analytic continuation, one computes the Fourier expansion of $F_{N,-n,s}(z)$.
	We then define the functions $j_{N,n}(z)$ by
	\begin{align}
j_{N,n}(z):=2\pi\sqrt{n}F_{N, -n}(z).\label{JNnDefinition}
	\end{align}
For $N=1$, we recover the $j_n(z)$ from the introduction up to the constant $2\pi \sqrt{n} c_1(n,0)=24 \sigma_1(n)$. \\

 In order to formally state Theorem \ref{AKNGeneralization} (4), for an arbitrary solution $a,b\in\Z$ to $ad-bc=1$, we define
\begin{align}\label{eqn:rcd}
r_z(c,d)&:=ac|z|^2+(ad+bc)\re(z)+bd,\\
\label{eqn:Lambdazdef}\Lambda_z&:=\left\{\alpha^2 |z|^2 + \beta \re(z) + \gamma^2>0: \alpha,\beta,\gamma\in \Z\right\},\\
\label{eqn:Slambdadef}S_{\lambda}&:=\left\{ (c,d)\in N\N_0\times \Z:\ \gcd(c,d)=1\text{ and }Q_{z}(c,d)=\lambda\right\},\\
\nonumber Q_{z}(c,d)&:=c^2|z|^2+2cd\re(z)+d^2.
\end{align}
Note that although $r_{z}(c,d)$ is not uniquely determined, $e(-nr_{z}(c,d)/Q_z(c,d))$ is well-defined. 

\begin{proof}[Sketch of proof of Theorem \ref{AKNGeneralization}]
Part (1) follows from the explicit Fourier expansions given in Proposition \ref{Niebur} and Proposition 2.1 of \cite{FirstPaper}. Parts (2) and (3) follow from a straightforward calculation. To prove part (4), we rewrite the claimed asymptotic formula in terms of the corresponding points $Mz$ with $M=\begin{psmallmatrix}a&b\\c&d\end{psmallmatrix} \in\Gamma_{\infty}\backslash \Gamma_0(N)$. Directly plugging in and simplifying yields that the claim in Theorem \ref{AKNGeneralization} (4) is equivalent to
\begin{equation*}
j_{N,n}(z)= \sum_{\substack{M\in \Gamma_{\infty}\backslash \Gamma_0(N)\\ n\im(Mz)\geq \im(z)}} e^{-2\pi in Mz}+O_{z}(n),
\end{equation*}
which can be shown using standard estimates and the Weil bound for Kloosterman sums.   
\end{proof}

The functions $H_{N,\varrho}^*$ do not have simple poles at $\varrho$, but they can still be seen as analogues of $H_{N,z}^*$ at the cusps of $\Gamma_0(N)$. 
It turns out that they are connected to the harmonic weight $2$ Eisenstein series $E_{2,N,\varrho}^*(\tau)$ for $\Gamma_0(N)$.  For $\re (s) >0$, define
\begin{equation*}
E_{2, N,\varrho,s}^*(\tau):=\sum_{M\in \Gamma_\varrho\setminus \Gamma_0(N)} j\left(M_\varrho M, \tau\right)^{-2}\left|j\left(M_\varrho M, \tau\right)\right|^{-2s}.
\end{equation*}
Using the Hecke trick, it is well-known (cf. Satz 6 of \cite{Hecke}) that $E_{2,N,\varrho,s}^*$ has an analytic continuation to $s=0$, denoted by $E_{2,N,\varrho}^*$. 

\begin{proof}[Sketch of proof of Theorem \ref{thm:HtauCusp}]
The Fourier expansion of $H_{N,\varrho}^*$ given in Theorem \ref{thm:HtauCusp} (1) follows from Lemma 5.4 of \cite{BK}. We obtain the statement by comparing to the Fourier expansion of $E_{2,N,\varrho}^*$ given in Theorem 1 of \cite{Smart}. For this we recall that the {\it generalized Kloosterman sums} for a cusp $\varrho$ of $\Gamma_0(N)$ are defined by
\begin{equation}\label{Krhodef}
	K_{i\infty, \varrho}(m,n;c):=\sum_{\left.\left.\left(\begin{smallmatrix} a & b \\ c & d\end{smallmatrix}\right) \in\Gamma_\infty\right\backslash \Gamma_0(N)M_\varrho\right/\Gamma_\infty^{\ell_\varrho}} e\left(\frac{m d}{\ell_\varrho c} +\frac{na}{c}\right).
\end{equation}
Parts (2) and (3) follow again from a straightforward calculation.
\end{proof}

\section{Further Remarks}\label{FinalRemarks}

\subsection{Applications to Formulas of Ramanujan}

Polar harmonic Maass forms can be used to solve the difficult problems of computing and estimating Fourier coefficients of meromorphic modular forms. Hardy and Ramanujan \cite{HR3} considered the special case where the form has a unique simple pole in $\mathrm{SL}_2(\Z)\backslash\H$. In particular, they found a formula for the reciprocal of the Eisenstein series $E_6$.  Ramanujan \cite{Ramanujan88} then conjectured further formulas for other examples of meromorphic modular forms, such as 
\[
\frac{1}{E_4(\tau)}=\sum_{n\geq0}\beta_ne^{2\pi in\tau}
,
\]
where $\beta_n$ is a sum of the form \eqref{eqn:Ramanujanlike} with $z=\rho:= \frac{1}{2}+i\sqrt{3}/2$.
Expansions like \eqref{eqn:Ramanujanlike} are very convenient since they turn out to converge extremely rapidly. It is easy to check that the main asymptotic growth in \eqref{eqn:Ramanujanlike} comes from the $\lambda=1$ terms, yielding
$$
\beta_n \sim (-1)^n\frac{3}{E_6(\rho)}e^{\pi n\sqrt{3}}.
$$
Such estimates confirm that the coefficients of meromorphic modular forms grow much faster than the coefficients of weakly holomorphic modular forms, as a lengthy calculation shows. Ramanujan's formulas for forms with simple poles were proven by Bialek in his Ph.D. thesis written under Berndt \cite{bia}.  Berndt, Bialek, and Yee \cite{bby} then pushed the Circle Method further to study examples with second-order poles, and they managed to prove the remaining cases of Ramanujan's formulas. The proofs of  \cite{bby, bia, HR3} all utilized a modification of the Hardy-Ramanujan Circle Method, but the calculations rapidly become more difficult with rising pole orders.\\

It turns out that, for $z\in\{i,\rho\}$, all of the meromorphic modular forms investigated by Ramanujan may be written as linear combinations of the series  
\begin{equation*}
f_{k,j,r}(z,\tau):= y^{-j}\sum_{m=0}^{\infty} \;\sideset{}{^*}\sum_{\mathfrak{b}\subseteq\Z[z]} \frac{C_{k}\left(\mathfrak{b},m\right)}{N(\mathfrak{b})^{\frac{k}{2}-j}} (4\pi m)^r e^{\frac{2\pi  mz}{N(\mathfrak{b})}} e^{2\pi i m \tau},
\end{equation*}
where $\mathfrak{b}$ runs over primitive ideals of $\Z[z]$, $N(\mathfrak{b})$ is the norm of $\mathfrak b$, and $C_k$ are certain functions on ideals which we next describe.  To be precise, for $\mathfrak{b}=(c\rho+d)\subset\Z[\rho]$,
we define
\begin{equation}\notag
C_{6m}\left(\mathfrak{b},n\right) :=\cos\left(\frac{\pi n}{N(\mathfrak{b})}\left(ad+bc-2ac-2bd\right) -6m\arctan\left(\frac{c\sqrt{3}}{2d-c}\right)\right),
\end{equation}
and we set $C_{m}(\mathfrak{b},n):=0$ if $6\nmid m$.
Similarly, for $\mathfrak{b}=(ci+d)\subseteq\Z[i]$,
we let
\begin{equation}\notag
C_{4m}\left(\mathfrak{b},n\right):=\cos\left(\frac{2\pi n}{N(\mathfrak{b})}\left(ac+bd\right)+4m\arctan\left(\frac{c}{d}\right)\right)
\end{equation}
and $C_{m}(\mathfrak{b},n):=0$ if $4\nmid m$. \\ 

To describe the coefficients in these linear combinations, we need a basis of meromorphic Poincar\'e series, originally discovered by Petersson. On the subgroup $\Gamma_0(N)$, a weight
$k\in\N_{>1}$, we define a  $2$-variable Poincar\'e series (for $k=1$, similar construction holds but an analytic continuation is required.) Setting 
\[
H_{2k,N}(z,\tau):=\sum_{M\in\Gamma_{\infty}\backslash\Gamma_0(N)}\frac{1}{1-e^{2\pi i(\tau-z)}}\bigg|_{2k,z}M,
\]
we then let
\[
\mathcal H_{2k,N}(z,\tau):=H_{2k}(z,\tau)+\sum_{r=0}^{2k-2}\frac{(2iv)^r}{r!}\frac{\partial^r}{\partial\overline{\tau}^r}H_{2k}(z,\overline{\tau}).
\]
In the variable $z$, this function is a weight $2k$ meromorphic cusp form on $\Gamma_0(N)$ with poles supported on $\H$, and in $\tau$ it is a polar harmonic Maass form in $\mathcal{H}_{2-2k}^{\operatorname{cusp}}(\Gamma_0(N))$. The following result, which follows by a careful computation and comparison of principal parts at poles on $\H$, shows that the functions $\mathcal H_{2k,N}$ provide bases for polar harmonic Maass forms. In what follows, $R_{k,\z}:=2i\frac{\partial}{\partial\z}+\frac{k}{\im(\z)}$ denotes the usual Maass raising operator with respect to the variable $\z$.
\begin{proposition}[Proposition 4.2 of \cite{BKFC2}]\label{PolarHMFExp}
If $f\in\mathcal{H}_{2-2k}^{\operatorname{cusp}}(\Gamma_0(N))$ with $k\in\N_{>1}$, then we can decompose $f$ as a sum of the form
\[
f(\tau)=\sum_{\ell=1}^r\sum_{n=0}^{n_{\ell}}a_{\ell,n}\left[R_{2k,z}^n\left(\mathcal H_{2k}(z,\tau)\right)\right]_{z=z_{\ell}}
,
\]
where $a_{\ell,n}\in\C$ and $z_1,\ldots,z_r$ are the locations of the poles of $f$ on $\H$. 
\end{proposition}

Proposition~\ref{PolarHMFExp} then allows us to describe the Fourier expansions of meromorphic cusp forms.
\begin{theorem}[Theorem 1.1 of \cite{BKFC2}]\label{Fouriercoefficients}
If $f$ is a meromorphic cusp form of integral weight $2-2k<0$ and $z\in\{i,\rho\}$ is the only pole of $f$ in $\SL_2(\Z)\setminus\mathbb{H}$, then, with coefficients $a_{\ell}$ defined by the expansion in Proposition~\ref{PolarHMFExp}, for $v>y$, we have the following Fourier expansion: 
$$
f(\tau) = 2e_{1,z}\sum_{n=0}^{n_0} a_{n} \sum_{j=0}^{n} \frac{(2k+n-1)!}{(2k+n-1-j)!}\binom{n}{j} f_{2k+2n,j,n-j}(z,\tau).
$$
\end{theorem}
\begin{remark}
A more general formula for meromorphic cusp forms with arbitrary order poles at arbitrary points is given in Theorem 4.1 of \cite{BK3}.
\end{remark}

\subsection{Applications to Green's Functions}
We now consider the functions 
\[
f_\mathcal{A}(\tau):=D^{\frac k2}\sum_{ Q\in\mathcal{A}}Q(\tau,1)^{-k}
,
\]
where $\mathcal{A}$ is an $\SL_2(\Z)$-equivalence class of integral binary quadratic forms of discriminant $-D<0$. In what follows, we denote by $\tau_Q$ the unique root of $Q(X,1)$ in the upper half plane.  These are analogous to Zagier's $f_{k,D}$ functions for $D>0$, which play an important role in Shimura and Shintani lifts. In the case when $D<0$ Bengoechea showed in her thesis \cite{Bengoechea13}  that one obtains interesting meromorphic modular forms. The importance of these functions comes to the forefront when one integrates them  which can be used to give evaluations of higher Green's functions at CM points \cite{BKvP}. Such Green's functions appeared in Gross and Zagier's \cite{GrossZagier86} evaluation of the local heights on $X_0(N)$ at Archimedean places and an identity between higher Green's functions evaluated at CM-points and the infinite part of the height pairing of CM-cycles was later established by Zhang \cite{Zhang}. We briefly recall the definition of these functions. 
\begin{definition}
For $k\in\mathbb{N}_{>1}$ and $N\in\N$, the {\it higher Green's function} $\mathcal G_k\colon \mathbb{H}\times\mathbb{H}\to\mathbb{C}$ is uniquely characterized by the following properties.
\begin{itemize}
\item [{\rm i)}] The function $\mathcal G_k$ is smooth and real-valued on $\mathbb{H}\times\mathbb{H}\setminus \{(z, \gamma z): \gamma\in\Gamma_0(N), z\in\mathbb{H}\}.$
\item [{\rm ii)}] For $\gamma_1, \gamma_2\in\Gamma_0(N)$, we have $\mathcal G_k(\gamma_1 z, \gamma_2\mathfrak z)= \mathcal G_k(z, \mathfrak z).$
\item[{\rm iii)}] We have
\[
\Delta_{0, z}\left(\mathcal G_k\left(z, \mathfrak z\right)\right)=k(1-k)\mathcal G_k\left(z,\mathfrak z\right)=
\Delta_{0, \mathfrak z}\left(\mathcal G_k\left(z, \mathfrak z\right)\right).
\]
\item[{\rm iv)}] As $z\to \mathfrak z$ 
\[
\mathcal G_k(z, \mathfrak z)=\frac{1}{e_{N,\mathfrak z}}\log\left(r_{\mathfrak z}\left(z\right)\right)+O(1),
\]
where
\[
r_{\mathfrak z}(z):=\operatorname{tanh}\left(\frac{d(z,\mathfrak z)}2\right)=|X_\mathfrak{z}(z)|.
\]
Here $d(z,\mathfrak z)$ denotes the hyperbolic distance between $z$ and $\mathfrak z$.
\item[{\rm v)}] As $z$ approaches a cusp, $\mathcal G_k(z, \mathfrak z)\to0$.
\end{itemize}
\end{definition}

These functions have a long history and Gross and Zagier conjectured \cite{GrossZagier86} that their evaluations at CM-points are essentially logarithms of algebraic numbers. Specifically, in the special case when the space of weight $2k$ cusp forms on $\Gamma_0(N)$ is trivial, their conjecture states that
\[
\mathcal G_k(z, \mathfrak z)=(D_1 D_2)^{\frac{1-k}{2}}\log(\alpha)
\]
for CM-points $z, \mathfrak z$ of discriminants $D_1 \text{ and } D_2$, respectively and $\alpha$ is an algebraic number. Various cases of this conjecture have been solved. For example Mellit, in his Ph.D. thesis \cite{Mellit}, proved the case $k=2, \mathfrak z=i$ and also gave an interpretation of $\alpha$ as a certain intersection number of certain higher Chow cycles. Viazovska \cite{Viazovska} then proved the conjecture when the two CM points lie in the same imaginary quadratic field. \\

The Petersson inner product of two $f_Q$ functions, suitably regularized and studied using the theory of polar harmonic Maass forms, gives evaluations of the Green's functions.
\begin{theorem}[Corollary 1.5 of \cite{BKvP}]\label{thm:innerGreens}
Suppose that $-D_1,-D_2<0$ are two discriminants.  If $Q_1, Q_2$ are quadratic forms in distinct $\operatorname{SL}_2(\Z)$-equivalence classes $\mathcal{A}_1, \mathcal{A}_2$ of discriminants $-D_1, -D_2$, then we have
$$
\left<f_{\mathcal{A}_1},f_{\mathcal{A}_2}\right>=\frac{(-1)^{k}\sqrt{\pi}\Gamma\left(k-\frac12\right) }{2^k (k-1)!}e_{1,\tau_{Q_1}}e_{1,\tau_{Q_2}}\mathcal G_k\left(\tau_{Q_1},\tau_{Q_2}\right)
$$
with the rational numbers $e_{1,z}$ defined in Remark 2 after Proposition \eqref{EllipticExpansion}.
\end{theorem}
\begin{remark}
It is possible that this result will shed new light on and offer an alternative approach to solving Gross and Zagier's conjecture.
\end{remark}

\end{document}